\documentclass[a4paper,draft]{amsproc}
\usepackage{eurosym}
\usepackage{amssymb,amsmath,mathrsfs,mathtools,amsfonts}
\usepackage[hyphens]{url}
\usepackage[T1]{fontenc}
\usepackage[utf8]{inputenc}

\theoremstyle{plain}
\newtheorem{thm}{Theorem}[section]
\newtheorem{prop}{Proposition}[section]
\newtheorem{lem}{Lemma}[section]
\newtheorem{cor}{Corollary}[section]
\theoremstyle{definition}

\theoremstyle{remark}

 \numberwithin{equation}{section}
\renewcommand{\leq}{\leqslant}

\subjclass[2010]{53C21, 53C25}
\keywords{Ricci-Bourguignon soliton, warped product manifold, sequential warped product manifold, Killing vector field, conformal vector field}

\begin{document}
\title[Ricci-Bourguignon solitons on sequential warped products]{%
RICCI-BOURGUIGNON SOLITONS ON SEQUENTIAL WARPED PRODUCT MANIFOLDS}
\author[A\c{c}\i kg\"{o}z Kaya]{\bfseries Dilek A\c{C}IKG\"{O}Z KAYA}
\author[\"{O}zg\"{u}r]{\bfseries Cihan \"{O}ZG\"{U}R}

\begin{abstract}
We study Ricci-Bourguignon solitons on sequential warped products. The
necessary conditions are obtained for a Ricci-Bourguignon soliton with the
structure of a sequential warped product to be an Einstein manifold when we
consider the potential field as a Killing or a conformal vector field.
\end{abstract}

\maketitle

\section{Introduction}

Let $(M,g)$ be a semi-Riemannian manifold and denote by $\mathrm{Ric}$ the
Ricci tensor of $(M,g)\mathrm{.}$ A semi-Riemannian manifold $(M,g)$ is said
to be a \textit{Ricci soliton }\cite{ham1a}, if there exists a smooth vector
field $X$ satisfying the equation
\begin{equation}
\mathrm{Ric}+\dfrac{1}{2}\mathcal{L}_{X}g=\lambda g  \label{Rs}
\end{equation}%
for some constant $\lambda $ and it is denoted by $(M,g,X,\lambda )$, where $%
\mathcal{L}$ denotes the Lie derivative, and the vector field $X\in
\mathfrak{X}(M)$ is called the \textit{potential vector field}. \

Ricci solitons are a natural generalization of Einstein manifolds. They
correspond to self-similar solutions of the Ricci flow equation
\begin{equation*}
\frac{\partial g}{\partial t}=-2\mathrm{Ric,}
\end{equation*}%
which was defined by Hamilton (\cite{ham1}, \cite{ham2}). Ricci solitons and
their some generalizations have been studied by many geometers in the recent
years. For example see (\cite{Bl-ozg}, \cite{Blaga}, \cite{cao1}, \cite%
{Chen-15}, \cite{Chen-16}, \cite{DesSoad}, \cite{emin}, \cite{lopez}, \cite%
{Guler}, \cite{pet2}, \cite{Sharma}) and the references therein.

If the potential vector field is the gradient of a smooth function $u$ on $M$%
, then $(M,g,\nabla u,\lambda )$ is called a \textit{gradient Ricci soliton}
and the equation (\ref{Rs}) turns into
\begin{equation*}
\mathrm{Ric}+\mathrm{Hess}u=\lambda g.
\end{equation*}

The study of the concept of the Ricci-Bourguignon soliton are introduced by
Dwivedi \cite{dwivedi}. They correspond to self-similar solutions of the
Ricci-Bourguignon flow equation
\begin{equation}
\frac{\partial g}{\partial t}=-2(\mathrm{Ric}-\rho Rg),  \label{RBflow}
\end{equation}%
where $R$ is the scalar curvature and $\rho \in \mathbb{R}$ is a constant.
The flow in equation (\ref{RBflow}) was introduced by Jean-Pierre
Bourguignon \cite{bourguignon}. Equation (\ref{RBflow}) is precisely the
Ricci flow for $\rho =0$. As in the Ricci flow case, the following
definition was given by Dwivedi \cite{dwivedi}.

A \textit{Ricci-Bourguignon soliton }(briefly $RBS$) is a semi-Riemannian
manifold $(M,g)$ endowed with a vector field $X$ on $M$ that satisfies
\begin{equation}
\mathrm{Ric}+\dfrac{1}{2}\mathcal{L}_{X}g=\lambda g+\rho Rg,
\label{RBsoliton}
\end{equation}%
where $\mathcal{L}_{X}g$ denotes the Lie derivative of the metric $g$ and $%
\lambda \in \mathbb{R}$ is a constant and it is denoted by $(M,g,X,\lambda
,\rho )$. If $X$ is the gradient of a smooth function $u$ on $M$, then $%
(M,g,\nabla u,\lambda ,\rho )$ is called a \textit{gradient
Ricci-Bourguignon soliton} and the equation \eqref{RBsoliton} turns into
\begin{equation*}
\mathrm{Ric}+\mathrm{Hess}u=\lambda g+\rho Rg.
\end{equation*}%
In \cite{dwivedi}, Dwivedi proved some results for the solitons of the
Ricci-Bourguignon flow, generalizing the corresponding results for Ricci
solitons. Later in \cite{soylu}, Y. Soylu gave classification theorems for
Ricci-Bourguignon solitons and almost solitons with concurrent potential
vector field. In \cite{ghosh}, A. Ghosh studied on Ricci-Bourguignon
solitons and Ricci-Bourguignon almost solitons on a Riemannian manifold and
proved some triviality results.

Warped product manifolds were defined by O'Neill and Bishop in \cite{obi} to
construct manifolds with negative curvature. They have an important role in
both geometry and physics. They are used in general relativity to model the
spacetime \cite{Chen}. Doubly, multiply and sequential warped product
manifolds are known generalizations of the warped product manifolds (\cite%
{dsu}, \cite{bum}, \cite{bud}). There are many papers in which Ricci
solitons on some Riemannian manifolds or on warped product manifolds or on
some generalizations of warped products have been studied, for example see (%
\cite{KO}, \cite{Blaga2}, \cite{dmsu}, \cite{FFG}, \cite{GO}, \cite{GOK},
\cite{kaoz}, \cite{kaoz2}, \cite{sahin}, \cite{ro}). By a motivation from
the above studies, in this paper, we consider Ricci-Bourguignon solitons on
sequential warped product manifolds which is an another generalization of
the warped product manifolds. By considering the potential vector field as a
Killing or a conformal vector field, we prove some results.

\section{Preliminaries}

Let $(M_{i},g_{i})$ be semi-Riemannian manifolds, $1\leq i\leq 3,$ and $%
f:M_{1}\longrightarrow \mathbb{R}^{+},$ $h:M_{1}\times M_{2}\longrightarrow
\mathbb{R}^{+}$ be two smooth functions. The \textit{sequential warped
product manifold } $M$ is the triple product manifold $M=(M_{1}\times
_{f}M_{2})\times _{h}M_{3}$ endowed with the metric tensor $g=(g_{1}\oplus
f^{2}g_{2})\oplus h^{2}g_{3}$ \cite{dsu}. Here the functions $f,h$ are
called the \textit{warping functions}.

Through out the paper, $(M,g)$ will be considered as a sequential warped
product manifold, where $M=M^{n}=(M_{1}^{n_{1}}\times
_{f}M_{2}^{n_{2}})\times _{h}M_{3}^{n_{3}}$ with the metric $g=(g_{1}\oplus
f^{2}g_{2})\oplus h^{2}g_{3}$. The restriction of the warping function $h:%
\overline{M}=M_{1}\times M_{2}\longrightarrow \mathbb{R}$ to $M_{1}\times
\{0\}$ is $h^{1}=h|_{M_{1}\times \{0\}}$.

We use the notation $\nabla $, $\nabla ^{i}$; $\mathrm{Ric}$, $\mathrm{Ric}%
^{i}$; $\mathrm{Hess}$, $\mathrm{Hess}^{i}$; $\Delta $, $\Delta ^{i}$; $%
\mathcal{L}$, $\mathcal{L}^{i}$ for the Levi-Civita connections, Ricci
tensors, Hessians, Laplacians and Lie derivatives of $M$, and $M_{i}$,
respectively. Hessian of $\overline{M}$ is denoted by $\overline{\mathrm{Hess%
}}$.

The following lemmas on sequential warped product manifolds are necessary to
prove our results.

\begin{lem}
\cite{dsu} 
Let $(M,g)$ be a sequential warped product and $X_{i},Y_{i}\in \mathfrak{X}%
(M_{i})$ for $1\leq i\leq 3$. Then\vspace{0.1cm}

\begin{enumerate}
\item $\nabla_{X_1}Y_1=\nabla _{X_1}^1Y_1$,

\item $\nabla _{X_{1}}X_{2}=\nabla _{X_{2}}X_{1}=X_{1}(lnf)X_{2}$,

\item $\nabla_{X_2}Y_2=\nabla_{X_2}^2Y_2-fg_2(X_2,Y_2)\nabla^{1}\!f$,

\item $\nabla _{X_{3}}X_{1}=\nabla _{X_{1}}X_{3}=X_{1}(lnh)X_{3}$,

\item $\nabla _{X_{2}}X_{3}=\nabla _{X_{3}}X_{2}=X_{2}(lnh)X_{3}$,

\item $\nabla _{X_{3}}Y_{3}=\nabla
_{X_{3}}^{3}Y_{3}-hg_{3}(X_{3},Y_{3})~\!\nabla h$.
\end{enumerate}
\end{lem}

\begin{lem}
\cite{dsu}\label{dsu:p23} Let $(M,g)$ be a sequential warped product and $%
X_{i},Y_{i}\in \mathfrak{X}(M_{i})$ for $1\leq i\leq 3$. Then

\begin{enumerate}
\item $\mathrm{Ric}(X_1,Y_1)=$ $\mathrm{Ric}^1(X_1,Y_1)-\frac{n_2}{f}\mathrm{%
Hess}^1\!f(X_1,Y_1)-\frac{n_3}{h}\overline{\mathrm{Hess}}{h}(X_1,Y_1)$,

\item $\mathrm{Ric}(X_2,Y_2)=\mathrm{Ric}^{2}(X_2,Y_2)-f^\sharp
g_2\left(X_2,Y_2\right)-\frac{n_3}{h}\overline{\mathrm{Hess}}h(X_2,Y_2)$,

\item $\mathrm{Ric}(X_3,Y_3)=\mathrm{Ric}^{3}(X_3,Y_3)-h^\sharp
g_3\left(X_3,Y_3\right)$,

\item $\mathrm{Ric}(X_{i},X_{j})=0$ when $i\neq j,$ where $f^{\sharp
}=\left( f\Delta ^{1}f+(n_{2}-1)\left\Vert \nabla ^{1}\!f\right\Vert
^{2}\right) $ and $h^{\sharp }=\left( h\Delta h+(n_{3}-1)\left\Vert \nabla
h\right\Vert ^{2}\right) $.
\end{enumerate}
\end{lem}

\begin{lem}
\cite{dsu}\label{dsu:p32} Let $(M,g)$ be a sequential warped product
manifold. A vector field $X\in \mathfrak{X}(M)$ satisfies the equation

\begin{eqnarray*}
\mathcal{L}_{X}g(Y,Z) &=&\left( \mathcal{L}_{X_{1}}^{1}g_{1}\right)
(Y_{1},Z_{1})+f^{2}\left( \mathcal{L}_{X_{2}}^{2}g_{2}\right)
(Y_{2},Z_{2})+h^{2}\left( \mathcal{L}_{X_{3}}^{3}g_{3}\right) (Y_{3},Z_{3})
\\
&&+2fX_{1}(f)g_{2}(Y_{2},Z_{2})+2h(X_{1}+X_{2})(h)g_{3}(Y_{3},Z_{3})
\end{eqnarray*}%
for $Y,Z\in \mathfrak{X}(M)$.
\end{lem}

A vector field $V$ on a Riemannian manifold $(M,g)$ is said to be \textit{%
conformal}, if there exists a smooth function on $M$ satisfying the equation
\begin{equation*}
\mathcal{L}_{V}g=2fg.
\end{equation*}%
If $f=0$, then $V$ is called a \textit{Killing} vector field.

\section{Main Results}

In this section, we examine the properties of Ricci-Bourguignon solitons on
sequential warped product manifolds.

Firstly we have the following theorem:

\begin{thm}
\label{teo:3.1} Let $M=(M_{1}\times _{f}M_{2})\times _{h}M_{3}$ be a
sequential warped product equipped with the metric $g=(g_{1}\oplus
f^{2}g_{2})\oplus h^{2}g_{3}.$ If $(M,g,X,\lambda ,\rho )$ is a $RBS$ with
potential vector field of the form $X=X_{1}+X_{2}+X_{3},$ where $X_{i}\in
\mathfrak{X}(M_{i})$ for $1\leq i\leq 3$, then

\begin{enumerate}
\item[(i)] $(M_{1},g_{1},X_{1},\lambda _{1},\rho _{1})$ is a $RBS$ when $%
\mathrm{Hess}f=\sigma g$ and $\overline{\mathrm{Hess}}h=\psi g,$ where ${%
\lambda _{1}+\rho _{1}R_{1}=\lambda +\rho R+\frac{n_{2}}{f}\sigma +\frac{%
n_{3}}{h}\psi }$.

\item[(ii)] $M_{2}$ is an Einstein manifold when $X_{2}$ a Killing vector
field and $\overline{\mathrm{Hess}}h=\psi g$.

\item[(iii)] $(M_{3},g_{3},h^{2}X_{3},\lambda _{3},\rho _{3})$ is a $RBS$,
where $\lambda _{3}+\rho _{3}R_{3}=\lambda h^{2}+\rho Rh^{2}+h^{\sharp
}-h(X_{1}+X_{2})(h)$.
\end{enumerate}
\end{thm}

\begin{proof}
Assume that $(M,g,X,\lambda ,\rho )$ is a $RBS$ with the structure of the
sequential warped product. Then for $Y,Z\in \chi (M)$, the equation
\begin{equation*}
Ric(Y,Z)+\frac{1}{2}\mathcal{L}_{X}g(Y,Z)=(\lambda +\rho R)g(Y,Z)
\end{equation*}%
is satisfied. Using Lemma \ref{dsu:p23} and Lemma \ref{dsu:p32} for vector
fields $Y$ and $Z$ such that $Y=Y_{1}+Y_{2}+Y_{3}$ and $Z=Z_{1}+Z_{2}+Z_{3}$%
, we have
\begin{eqnarray}
&&\mathrm{Ric}^{1}(Y_{1},Z_{1})-\frac{n_{2}}{f}\mathrm{Hess}%
^{1}f(Y_{1},Z_{1})-\frac{n_{3}}{h}\overline{\mathrm{Hess}}h(Y_{1},Z_{1})
\notag  \label{warpRB} \\
&&+\mathrm{Ric}^{2}(Y_{2},Z_{2})-f^{\sharp }g_{2}(Y_{2},Z_{2})-\frac{n_{3}}{h%
}\overline{\mathrm{Hess}}h(Y_{2},Z_{2}) \\
&&+\mathrm{Ric}^{3}(Y_{3},Z_{3})-h^{\sharp }g_{3}(Y_{3},Z_{3})\hspace{2cm}
\notag \\
&&+\dfrac{1}{2}\mathcal{L}_{X_{1}}^{1}g_{1}(Y_{1},Z_{1})+\dfrac{1}{2}f^{2}%
\mathcal{L}_{X_{2}}^{2}g_{2}(Y_{2},Z_{2})+\dfrac{1}{2}h^{2}\mathcal{L}%
_{X_{3}}^{3}g_{3}(Y_{3},Z_{3})  \notag \\
&&+fX_{1}(f)g_{2}(Y_{2},Z_{2})+h(X_{1}+X_{2})(h)g_{3}(Y_{3},Z_{3})\hspace{2cm%
}  \notag \\
&=&(\lambda +\rho R)g_{1}(Y_{1},Z_{1})+(\lambda +\rho
R)f^{2}g_{2}(Y_{2},Z_{2})+(\lambda +\rho R)h^{2}g_{3}(Y_{3},Z_{3}).  \notag
\end{eqnarray}%
Let $Y=Y_{1}$ and $Z=Z_{1}$. So from the equation (\ref{warpRB}), if $%
\mathrm{Hess}f=\sigma g$ and $\overline{\mathrm{Hess}}h=\psi g,$ then we get
\begin{eqnarray}
\mathrm{Ric}^{1}(Y_{1},Z_{1})+\dfrac{1}{2}\mathcal{L}%
_{X_{1}}^{1}g_{1}(Y_{1},Z_{1}) &=&\lambda _{1}g_{1}(Y_{1},Z_{1})+[-\lambda
_{1}+\lambda +\rho R+\frac{n_{2}}{f}\sigma +\frac{n_{3}}{h}\psi
]g_{1}(Y_{1},Z_{1})  \notag \\
&=&\lambda _{1}g_{1}(Y_{1},Z_{1})+\rho _{1}R_{1}g_{1}(Y_{1},Z_{1}).  \notag
\end{eqnarray}%
Hence $(M_{1},g_{1},X_{1},\lambda _{1},\rho _{1})$ is a $RBS$, where ${%
\lambda _{1}+\rho _{1}R_{1}=\lambda +\rho R+\frac{n_{2}}{f}\sigma +\frac{%
n_{3}}{h}\psi .}$ \vspace{0.1cm}

Now, let $Y=Y_{2}$ and $Z=Z_{2}$. Then
\begin{eqnarray*}
&&\mathrm{Ric}^{2}(Y_{2},Z_{2})-f^{\sharp }g_{2}(Y_{2},Z_{2})-\frac{n_{3}}{h}%
\overline{\mathrm{Hess}}h(Y_{2},Z_{2})+\dfrac{1}{2}f^{2}\mathcal{L}%
_{X_{2}}^{2}g_{2}(Y_{2},Z_{2})+fX_{1}(f)g_{2}(Y_{2},Z_{2}) \\
&&\hspace{4cm}=(\lambda +\rho R)f^{2}g_{2}(Y_{2},Z_{2}).
\end{eqnarray*}%
Here, if $X_{2}$ is a Killing vector field and $\overline{\mathrm{Hess}}%
h=\psi g$, we get
\begin{equation*}
\mathrm{Ric}^{2}(Y_{2},Z_{2})=(\lambda f^{2}+\rho Rf^{2}+f^{\sharp }+\frac{%
n_{3}}{h}\psi f^{2}-fX_{1}(f))g_{2}(Y_{2},Z_{2}),
\end{equation*}%
which implies that $M_{2}$ is an Einstein manifold.

Finally, let $Y=Y_{3}$ and $Z=Z_{3}$. Then%
\begin{eqnarray*}
&&\mathrm{Ric}^{3}(Y_{3},Z_{3})+\dfrac{1}{2}\mathcal{L}%
_{h^{2}X_{3}}^{3}g_{3}(Y_{3},Z_{3})\hspace{8cm} \\
&=&\lambda _{3}g_{3}(Y_{3},Z_{3})+[-\lambda _{3}+\lambda h^{2}+\rho
Rh^{2}+h^{\sharp }-h(X_{1}+X_{2})(h)]g_{3}(Y_{3},Z_{3})\hspace{0cm} \\
&=&\lambda _{3}g_{3}(Y_{3},Z_{3})+\rho _{3}R_{3}g_{3}(Y_{3},Z_{3}),\hspace{%
5.7cm}
\end{eqnarray*}%
which means that $(M_{3},g_{3},h^{2}X_{3},\lambda _{3},\rho _{3})$ is a $RBS$%
, where $\lambda _{3}+\rho _{3}R_{3}=\lambda h^{2}+\rho Rh^{2}+h^{\sharp
}-h(X_{1}+X_{2})(h)$.
\end{proof}

In the following theorems, we provide some conditions for the manifolds $%
M_{i},$ $(1\leq i\leq 3)$ to be Einstein manifolds.

\begin{thm}
Let $M=(M_{1}\times _{f}M_{2})\times _{h}M_{3}$ be a sequential warped
product equipped with the metric $g=(g_{1}\oplus f^{2}g_{2})\oplus
h^{2}g_{3} $. If $(M,g,X,\lambda ,\rho )$ is a $RBS$ and $X$ is a Killing
vector field, then

\begin{enumerate}
\item[(i)] $M_{1}$ is an Einstein manifold when $\mathrm{Hess}f=\sigma g$
and $\overline{\mathrm{Hess}}h=\psi g$.

\item[(ii)] $M_{2}$ is an Einstein manifold when $\overline{\mathrm{Hess}}%
h=\psi g$.

\item[(iii)] $M_{3}$ is an Einstein manifold.
\end{enumerate}
\end{thm}

\begin{proof}
Let $(M,g,X,\lambda ,\rho )$ be a $RBS$ with the structure of the sequential
warped product and $X$ a Killing vector field. Then for all $Y,Z\in \chi (M)$%
, we have $Ric(Y,Z)=(\lambda +\rho R)g(Y,Z)$. From equation (\ref{warpRB}),
we may write
\begin{equation*}
\mathrm{Ric}^{1}(Y_{1},Z_{1})=(\lambda +\rho R+\frac{n_{2}}{f}\sigma +\frac{%
n_{3}}{h}\psi ]g_{1}(Y_{1},Z_{1})
\end{equation*}%
\begin{equation*}
\mathrm{Ric}^{2}(Y_{2},Z_{2})=(\lambda f^{2}+\rho Rf^{2}+f^{\sharp }+\frac{%
n_{3}}{h}\psi f^{2})g_{2}(Y_{2},Z_{2}).
\end{equation*}%
and
\begin{equation*}
\mathrm{Ric}^{3}(Y_{3},Z_{3})=(\lambda h^{2}+\rho Rh^{2}+h^{\sharp
})g_{3}(Y_{3},Z_{3}),
\end{equation*}%
which imply that $M_{1}$, $M_{2}$ and $M_{3}$ are Einstein manifolds.
\end{proof}

\begin{thm}
\label{teo:3.3} Let $M=(M_{1}\times _{f}M_{2})\times _{h}M_{3}$ be a
sequential warped product equipped with the metric $g=(g_{1}\oplus
f^{2}g_{2})\oplus h^{2}g_{3}$ and $(M,g,X,\lambda ,\rho )$ a $RBS$. Assume
that $\mathrm{Hess}f=\sigma g$ and $\overline{\mathrm{Hess}}h=\psi g$. Then $%
M_{i}$ $(1\leq i\leq 3)$ are Einstein manifolds if one of the following
conditions hold:

\begin{enumerate}
\item[(i)] $X=X_{1}$ and $X_{1}$ is Killing on $M_{1}$.

\item[(ii)] $X=X_{2}$ and $X_{2}$ is Killing on $M_{2}$.

\item[(iii)] $X=X_{3}$ and $X_{3}$ is Killing on $M_{3}$.
\end{enumerate}
\end{thm}

\begin{proof}
Let $(M,g,X,\lambda ,\rho )$ be a $RBS$ with the structure of the sequential
warped product. Assume that $\mathrm{Hess}f=\sigma g$ and $\overline{\mathrm{%
Hess}}h=\psi g$. If $X=X_{1}$ and $X_{1}$ is Killing on $M_{1}$, using Lemma %
\ref{dsu:p32} we have
\begin{equation*}
\mathcal{L}_{X}g=2fX_{1}(f)g_{2}.
\end{equation*}%
So by using of above equation in (\ref{warpRB}), we get
\begin{equation*}
\mathrm{Ric}^{1}(Y_{1},Z_{1})=(\lambda +\rho R+\frac{n_{2}}{f}\sigma +\frac{%
n_{3}}{h}\psi ]g_{1}(Y_{1},Z_{1})
\end{equation*}%
\begin{equation*}
\mathrm{Ric}^{2}(Y_{2},Z_{2})=(\lambda f^{2}+\rho Rf^{2}+f^{\sharp }+\frac{%
n_{3}}{h}\psi f^{2}-fX_{1}(f))g_{2}(Y_{2},Z_{2}).
\end{equation*}%
and
\begin{equation*}
\mathrm{Ric}^{3}(Y_{3},Z_{3})=(\lambda h^{2}+\rho Rh^{2}+h^{\sharp
})g_{3}(Y_{3},Z_{3}).
\end{equation*}%
Thus the manifolds $M_{1}$, $M_{2}$ and $M_{3}$ are Einstein. Using the same
pattern, $(ii)$ and $(iii)$ can be verified.
\end{proof}

\begin{thm}
\label{teo:3.4} Let $M=(M_{1}\times _{f}M_{2})\times _{h}M_{3}$ be a
sequential warped product equipped with the metric $g=(g_{1}\oplus
f^{2}g_{2})\oplus h^{2}g_{3}$, $(M,g,X,\lambda ,\rho )$ a $RBS$ and $X$ a
conformal vector field. Then,

\begin{enumerate}
\item[(i)] $M_{1}$ is an Einstein manifold when$\ \mathrm{Hess}f=\sigma g$
and $\overline{\mathrm{Hess}}h=\psi g$.

\item[(ii)] $M_{2}$ is an Einstein manifold when $\overline{\mathrm{Hess}}%
h=\psi g$.

\item[(iii)] $M_{3}$ is an Einstein manifold.
\end{enumerate}
\end{thm}

\begin{proof}
Assume that $(M,g,X,\lambda ,\rho )$ is a $RBS$ with the structure of the
sequential warped product and $X$ is a conformal vector field with factor $%
2\alpha $. Then $\alpha $ is a constant and
\begin{equation*}
\mathrm{Ric}(Y,Z)=(\lambda +\rho R-\alpha )g(Y,Z).
\end{equation*}%
Then using (\ref{warpRB}), the above equation implies
\begin{eqnarray*}
&&\mathrm{Ric}^{1}(Y_{1},Z_{1})-\frac{n_{2}}{f}\mathrm{Hess}%
^{1}f(Y_{1},Z_{1})-\frac{n_{3}}{h}\overline{\mathrm{Hess}}h(Y_{1},Z_{1})+%
\mathrm{Ric}^{2}(Y_{2},Z_{2})-f^{\sharp }g_{2}(Y_{2},Z_{2})\hspace{5cm} \\
&&-\frac{n_{3}}{h}\overline{\mathrm{Hess}}h(Y_{2},Z_{2})+\mathrm{Ric}%
^{3}(Y_{3},Z_{3})-h^{\sharp }g_{3}(Y_{3},Z_{3})\hspace{11cm} \\
&=&(\lambda +\rho R-\alpha )g_{1}(Y_{1},Z_{1})+(\lambda +\rho R-\alpha
)f^{2}g_{2}(Y_{2},Z_{2})+(\lambda +\rho R-\alpha )h^{2}g_{3}(Y_{3},Z_{3})%
\hspace{5cm}.
\end{eqnarray*}%
If $\mathrm{Hess}f=\sigma g$ and $\overline{\mathrm{Hess}}h=\psi g$, then we
get

$\mathrm{Ric}^{1}(Y_{1},Z_{1})={(\lambda +\rho R-\alpha +\frac{n_{2}}{f}%
\sigma +\frac{n_{3}}{h}\psi )}g_{1}(Y_{1},Z_{1})$

$\mathrm{Ric}^{2}(Y_{2},Z_{2})={(\lambda f^{2}+\rho Rf^{2}-\alpha f^{2}+%
\frac{n_{3}}{h}\psi f^{2}+f^{\sharp })}g_{2}(Y_{2},Z_{2})$

$\mathrm{Ric}^3(Y_3,Z_3)=(\lambda h^{2} +\rho R h^{2}- \alpha
h^{2}+h^{\sharp})g_{3}(Y_{3},Z_{3})$.

Hence, $M_{1}$, $M_{2}$ and $M_{3}$ are Einstein manifolds.
\end{proof}

Using Lemma \ref{dsu:p32} we can state the following theorem:

\begin{thm}
\label{teo:3.5} Let $M=(M_{1}\times _{f}M_{2})\times _{h}M_{3}$ be a
sequential warped product equipped with the metric $g=(g_{1}\oplus
f^{2}g_{2})\oplus h^{2}g_{3}.$ Then $(M,g,X,\lambda ,\rho )$ is Einstein if
one of the following conditions hold:

\begin{enumerate}
\item[(i)] $X=X_3$ and $X_3$ is a Killing vector field on $M_3$.

\item[(ii)] $X_1$ is a Killing vector field on $M_1$, $X_2$ and $X_3$ are
conformal vector fields on $M_2$ and $M_3$ with factors $-2X_1(\ln f)$ and $%
-2(X_1+X_2)(\ln h)$, respectively.

\item[(iii)] $X=X_{2}+X_{3},$ $X_{2}$ and $X_{3}$ are Killing on $M_{2}$ and
$M_{3}$, respectively and $X_{2}(h)=0$.
\end{enumerate}
\end{thm}

The next theorem gives the necessary condition for components of the vector
field $X$ to be a conformal vector field.

\begin{thm}
\label{teo:3.6} Let $M=(M_{1}\times _{f}M_{2})\times _{h}M_{3}$ be a
sequential warped product equipped with the metric $g=(g_{1}\oplus
f^{2}g_{2})\oplus h^{2}g_{3}$ and $(M,g,X,\lambda ,\rho )$ a $RBS.$

\begin{enumerate}
\item[(i)] If $M_{1}$ is an Einstein manifold, $\mathrm{Hess}f=\sigma g$ and
$\overline{\mathrm{Hess}}h=\psi g$ then $X_{1}$ is a conformal vector field
on $M_{1}$.

\item[(ii)] If $M_{2}$ is an Einstein manifold and $\overline{\mathrm{Hess}}%
h=\psi g,$ then $X_{2}$ is a conformal vector field on $M_{2}$.

\item[(iii)] If $M_{3}$ is an Einstein manifold, then $X_{3}$ is a conformal
vector field on $M_{3}$.
\end{enumerate}
\end{thm}

\begin{proof}
Let $(M_{1},g_{1})$, $(M_{2},g_{2})$ and $(M_{3},g_{3})$ be Einstein
manifolds with factors $\mu _{1}$, $\mu _{2}$ and $\mu _{3}$, respectively
and $(M,g,X,\lambda ,\rho )$ a $RBS$ with the structure of the sequential
warped product. If $\mathrm{Hess}f=\sigma g$ and $\overline{\mathrm{Hess}}%
h=\psi g$, then from the equation (\ref{warpRB}), we get
\begin{eqnarray*}
&&\mu _{1}g_{1}(Y_{1},Z_{1})-\frac{n_{2}}{f}\sigma g_{1}(Y_{1},Z_{1})-\frac{%
n_{3}}{h}\psi g_{1}(Y_{1},Z_{1})+\mu _{2}g_{2}(Y_{2},Z_{2})-f^{\sharp
}g_{2}(Y_{2},Z_{2}) \\
&&-\frac{n_{3}}{h}\psi f^{2}g_{2}(Y_{2},Z_{2})+\mu
_{3}g_{3}(Y_{3},Z_{3})-h^{\sharp }g_{3}(Y_{3},Z_{3})+\dfrac{1}{2}\mathcal{L}%
_{X_{1}}^{1}g_{1}(Y_{1},Z_{1}) \\
&&+\dfrac{1}{2}f^{2}\mathcal{L}_{X_{2}}^{2}g_{2}(Y_{2},Z_{2})+\dfrac{1}{2}%
h^{2}\mathcal{L}_{X_{3}}^{3}g_{3}(Y_{3},Z_{3}) \\
&&+fX_{1}(f)g_{2}(Y_{2},Z_{2})+h(X_{1}+X_{2})(h)g_{3}(Y_{3},Z_{3}) \\
&=&(\lambda +\rho R)g_{1}(Y_{1},Z_{1})+(\lambda +\rho
R)f^{2}g_{2}(Y_{2},Z_{2})+(\lambda +\rho R)h^{2}g_{3}(Y_{3},Z_{3}).
\end{eqnarray*}%
Thus,
\begin{equation*}
\mathcal{L}_{X_{1}}^{1}g_{1}(Y_{1},Z_{1})=2(\lambda +\rho R-\mu _{1}+\frac{%
n_{2}}{f}\sigma +\frac{n_{3}}{h}\psi )g_{1}(Y_{1},Z_{1})
\end{equation*}%
\begin{equation*}
\mathcal{L}_{X_{2}}^{2}g_{2}(Y_{2},Z_{2})=\frac{2}{f^{2}}(\lambda f^{2}+\rho
Rf^{2}-\mu _{2}+f^{\sharp }+\frac{n_{3}}{h}\psi
f^{2}-fX_{1}(f))g_{2}(Y_{2},Z_{2}).
\end{equation*}%
and
\begin{equation*}
\mathcal{L}_{X_{3}}^{3}g_{3}(Y_{3},Z_{3})=\frac{2}{h^{2}}(\lambda h^{2}+\rho
Rh^{2}-\mu _{3}+h^{\sharp }-h(X_{1}+X_{2})(h))g_{3}(Y_{3},Z_{3}).
\end{equation*}%
Hence, $X_{1}$, $X_{2}$ and $X_{3}$ are conformal vector fields on $M_{1}$, $%
M_{2}$ and $M_{3}$, respectively.
\end{proof}

\begin{thm}
\label{teo:3.7} Let $M=(M_{1}\times _{f}M_{2})\times _{h}M_{3}$ be a
sequential warped product equipped with the metric $g=(g_{1}\oplus
f^{2}g_{2})\oplus h^{2}g_{3}$ and $(M,g,X,\lambda ,\rho )$ a $RBS$ such that
$X=\nabla u$. Then

\begin{enumerate}
\item[(i)] $(M_{1},g_{1},\nabla \phi _{1},\lambda _{1},\rho _{1})$ is a
gradient $RBS$ when $\phi _{1}=u_{1}-n_{2}\ln f-n_{3}\ln h_{1}$ and $u_{1}=u$%
, where $\lambda _{1}+\rho _{1}R_{1}=\lambda +\rho R$.

\item[(ii)] $(M_{3},g_{3},\nabla \phi _{3},\lambda _{3},\rho _{3})$ is a
gradient $RBS$ when $\phi _{3}=u$, where $\lambda _{3}+\rho
_{3}R_{3}=\lambda h^{2}+\rho Rh^{2}+h^{\sharp }$.
\end{enumerate}
\end{thm}

\begin{proof}
Assume that $(M,g,X,\lambda ,\rho )$ is a $RBS$ with the structure of the
sequential warped product such that $X=\nabla u$. Then for $Y,Z\in \chi (M)$
\begin{equation}
\mathrm{Ric}(Y,Z)+\mathrm{Hess}u(Y,Z)=\lambda g(Y,Z)+\rho Rg(Y,Z)
\label{gradientRB}
\end{equation}%
is satisfied. Now let $Y=Y_{1}$ and $Z=Z_{1}$. Then the equation (\ref%
{gradientRB}) becomes
\begin{eqnarray*}
&&\mathrm{Ric}^{1}(Y_{1},Z_{1})-\frac{n_{2}}{f}\mathrm{Hess}%
^{1}f(Y_{1},Z_{1})-\frac{n_{3}}{h}\overline{\mathrm{Hess}}h(Y_{1},Z_{1})+%
\mathrm{Hess}u_{1}(Y_{1},Z_{1}) \\
&=&\lambda g_{1}(Y_{1},Z_{1})+\rho Rg_{1}(Y_{1},Z_{1})
\end{eqnarray*}%
or equivalently%
\begin{eqnarray*}
\mathrm{Ric}^{1}(Y_{1},Z_{1})+\mathrm{Hess}\phi _{1}(Y_{1},Z_{1}) &=&\lambda
_{1}g_{1}(Y_{1},Z_{1})+(-\lambda _{1}+\lambda +\rho R)g_{1}(Y_{1},Z_{1}) \\
&=&\lambda _{1}g_{1}(Y_{1},Z_{1})+\rho _{1}R_{1}g_{1}(Y_{1},Z_{1}),
\end{eqnarray*}%
where $\phi _{1}=u_{1}-n_{2}\ln f-n_{3}\ln h_{1}$ and $u_{1}=u$. In this
case, $(M_{1},g_{1},\nabla \phi _{1},\lambda _{1},\rho _{1})$ is a gradient $%
RBS$ soliton, where $\lambda _{1}+\rho _{1}R_{1}=\lambda +\rho R$. Using the
same pattern, $(ii)$ can be verified.
\end{proof}

\section{Ricci-Bourguignon Solitons on Sequential Warped Product Space-Times}

In this section, we will examine Ricci-Bourguignon solitons admitting two
well-known space-times, namely standard static space-times and generalized
Robertson-Walker space-times.

Let $(M_{i},g_{i})$ be semi-Riemannian manifolds, $1\leq i\leq 2,$ and $%
f:M_{1}\longrightarrow \mathbb{R}^{+},$ $h:M_{1}\times M_{2}\longrightarrow
\mathbb{R}^{+}$ two smooth functions. The $(n_{1}+n_{2}+1)$- dimensional
\textit{sequential standard static space-time }\cite{dsu} $\overline{M}$ is
the triple product manifold $\overline{M}=(M_{1}\times _{f}M_{2})\times
_{h}I $ endowed with the metric tensor $\overline{g}=(g_{1}\oplus
f^{2}g_{2})\oplus h^{2}(-dt^{2})$. Here $I$ is an open, connected
subinterval of $\mathbb{R}$ and $dt^{2}$ is the usual Euclidean metric
tensor on $I$.

\begin{prop}
\cite{dsu} \label{dsu:p47} Let $(\overline{M}=(M_{1}\times_{f}M_{2})\times
_{h}I,\overline{g})$ be a sequential standard static space-time and $%
X_{i},Y_{i}\in \mathfrak{X}(M_{i})$ for $1\leq i\leq 2$. Then

\begin{enumerate}
\item $\overline{\nabla}_{X_1}Y_1=\nabla _{X_1}^1Y_1$,

\item $\overline{\nabla }_{X_{1}}X_{2}=\overline{\nabla}
_{X_{2}}X_{1}=X_{1}(\ln f)X_{2}$,

\item $\overline{\nabla}_{X_2}Y_2=\nabla_{X_2}^2Y_2-fg_2(X_2,Y_2)\nabla^{1}%
\!f$,

\item $\overline{\nabla} _{X_{i}}\partial_{t}=\overline{\nabla}
_{\partial_{t}}X_{i}=X_{i}(\ln h)\partial_{t}$, $i=1,2$

\item $\overline{\nabla} _{\partial_{t}}\partial_{t}= h grad h$,
\end{enumerate}
\end{prop}

\begin{prop}
\cite{dsu}\label{dsu:p49} Let $(\overline{M}=(M_{1}\times_{f}M_{2})\times
_{h}I,\overline{g})$ be a sequential standard static space-time and $%
X_{i},Y_{i}\in \mathfrak{X}(M_{i})$ for $1\leq i\leq 2$. Then

\begin{enumerate}
\item $\overline{\mathrm{Ric}}(X_1,Y_1)=$ $\mathrm{Ric}^1(X_1,Y_1)-\frac{n_2%
}{f}\mathrm{Hess}^1\!f(X_1,Y_1)-\frac{1}{h}\overline{\mathrm{Hess}}{h}%
(X_1,Y_1)$,

\item $\overline{\mathrm{Ric}}(X_2,Y_2)=\mathrm{Ric}^{2}(X_2,Y_2)-f^\sharp
g_2\left(X_2,Y_2\right)-\frac{1}{h}\overline{\mathrm{Hess}}h(X_2,Y_2)$,

\item $\overline{\mathrm{Ric}}(\partial_{t},\partial_{t})=h \Delta h$,

\item $\overline{\mathrm{Ric}}(X_{i},Y_{j})=0$ when $i\neq j,$ where $%
f^{\sharp }=\left( f\Delta ^{1}f+(n_{2}-1)\left\Vert \nabla
^{1}\!f\right\Vert ^{2}\right) $.
\end{enumerate}
\end{prop}

By using of Lemma \ref{dsu:p32}, it is easy to state the following Corollary:

\begin{cor}
\label{ssst} Let $(\overline{M}=(M_{1}\times _{f}M_{2})\times _{h}I,%
\overline{g})$ be a sequential standard static space-time. Then
\begin{eqnarray*}
\mathcal{L}_{\overline{X}}\overline{g}(\overline{Y},\overline{Z}) &=&\left(
\mathcal{L}_{X_{1}}^{1}g_{1}\right) (Y_{1},Z_{1})+f^{2}\left( \mathcal{L}%
_{X_{2}}^{2}g_{2}\right) (Y_{2},Z_{2})-2h^{2}uv\frac{\partial w}{\partial t}
\\
&&+2fX_{1}(f)g_{2}(Y_{2},Z_{2})-2uvh(X_{1}+X_{2})(h),
\end{eqnarray*}%
where $\overline{X}=X_{1}+X_{2}+w\partial _{t}$, $\overline{Y}%
=Y_{1}+Y_{2}+u\partial _{t}$ , $\overline{Z}=Z_{1}+Z_{2}+v\partial _{t}\in
\chi (\overline{M})$.
\end{cor}

Now we consider a $RBS$\ with the structure of the sequential standard
static space-times. By using Theorem \ref{teo:3.1}, the following result can
be given:

\begin{thm}
\label{teo:4.1} Let $\overline{M}=(M_{1}\times _{f}M_{2})\times _{h}I$ be a
sequential standard static space-time equipped with the metric $\overline{g}%
=(g_{1}\oplus f^{2}g_{2})\oplus h^{2}(-dt^{2}).$ If $(\overline{M},\overline{%
g},\overline{X},\overline{\lambda },\overline{\rho })$ is a $RBS$ with $%
\overline{X}=X_{1}+X_{2}+w\partial _{t},$ where $X_{i}\in \mathfrak{X}%
(M_{i}) $ for $1\leq i\leq 2$ and $w\partial _{t}\in \chi (I)$, then

\begin{enumerate}
\item[(i)] $(M_{1},g_{1},X_{1},\lambda _{1},\rho _{1})$ is a $RBS$ when $%
\mathrm{Hess}f=\sigma \overline{g}$ and $\overline{\mathrm{Hess}}h=\psi
\overline{g},$ where ${\lambda _{1}+\rho _{1}R_{1}=\overline{\lambda }+%
\overline{\rho }\overline{R}+\frac{n_{2}}{f}\sigma +\frac{1}{h}\psi }$.

\item[(ii)] $M_{2}$ is an Einstein manifold when $X_{2}$ a Killing vector
field and $\overline{\mathrm{Hess}}h=\psi \overline{g}$.

\item[(iii)] ${-\frac{\Delta h}{h}+\frac{\partial w}{\partial t}+\frac{1}{h}%
(X_{1}+X_{2})(h)=\overline{\lambda }+\overline{\rho }\overline{R}}$.
\end{enumerate}
\end{thm}

\begin{proof}
Let $(\overline{M},\overline{g},\overline{X},\overline{\lambda },\overline{%
\rho })$ be a $RBS$ with the structure of the sequential warped product.
Then for $\overline{Y},\overline{Z}\in \chi (\overline{M})$, the equation
\begin{equation*}
\overline{\mathrm{Ric}}(\overline{Y},\overline{Z})+\frac{1}{2}\mathcal{L}_{%
\overline{X}}\overline{g}(\overline{Y},\overline{Z})=(\overline{\lambda }+%
\overline{\rho }\overline{R})\overline{g}(Y,Z)
\end{equation*}%
is satisfied. Using Proposition \ref{dsu:p49} and Corollary \ref{ssst} for
vector fields $\overline{Y}=Y_{1}+Y_{2}+u\partial _{t}$ and $\overline{Z}%
=Z_{1}+Z_{2}+v\partial _{t}$, we get
\begin{eqnarray}
&&\mathrm{Ric}^{1}(Y_{1},Z_{1})-\frac{n_{2}}{f}\mathrm{Hess}%
^{1}f(Y_{1},Z_{1})-\frac{1}{h}\overline{\mathrm{Hess}}h(Y_{1},Z_{1})  \notag
\label{eq:41} \\
&&+\mathrm{Ric}^{2}(Y_{2},Z_{2})-f^{\sharp }g_{2}(Y_{2},Z_{2})-\frac{1}{h}%
\overline{\mathrm{Hess}}h(Y_{2},Z_{2}) \\
&&+h\Delta huv\hspace{2cm}  \notag \\
&&+\dfrac{1}{2}\mathcal{L}_{X_{1}}^{1}g_{1}(Y_{1},Z_{1})+\dfrac{1}{2}f^{2}%
\mathcal{L}_{X_{2}}^{2}g_{2}(Y_{2},Z_{2})-h^{2}\frac{\partial w}{\partial t}%
uv  \notag \\
&&+fX_{1}(f)g_{2}(Y_{2},Z_{2})-uvh(X_{1}+X_{2})(h)\hspace{2cm}  \notag \\
&=&(\overline{\lambda }+\overline{\rho }\overline{R})g_{1}(Y_{1},Z_{1})+(%
\overline{\lambda }+\overline{\rho }\overline{R})f^{2}g_{2}(Y_{2},Z_{2})-(%
\overline{\lambda }+\overline{\rho }\overline{R})h^{2}uv.  \notag
\end{eqnarray}%
When the arguments are restricted to the factor manifolds, we obtain%
\begin{equation*}
\mathrm{Ric}^{1}(Y_{1},Z_{1})-\frac{n_{2}}{f}\sigma g_{1}(Y_{1},Z_{1})-\frac{%
1}{h}\psi g_{1}(Y_{1},Z_{1})+\dfrac{1}{2}\mathcal{L}%
_{X_{1}}^{1}g_{1}(Y_{1},Z_{1})
\end{equation*}%
\begin{equation}  \label{eq:42}
=(\overline{\lambda }+\overline{\rho }\overline{R})g_{1}(Y_{1},Z_{1}),
\end{equation}

\begin{equation*}
\mathrm{Ric}^{2}(Y_{2},Z_{2})-f^{\sharp }g_{2}(Y_{2},Z_{2})-\frac{1}{h}%
\overline{\mathrm{Hess}}h(Y_{2},Z_{2})+\dfrac{1}{2}f^{2}\mathcal{L}%
_{X_{2}}^{2}g_{2}(Y_{2},Z_{2})+fX_{1}(f)g_{2}(Y_{2},Z_{2})
\end{equation*}%
\begin{equation}  \label{eq:43}
=(\overline{\lambda }+\overline{\rho }\overline{R})f^{2}g_{2}(Y_{2},Z_{2}).
\end{equation}

and
\begin{equation}  \label{eq:44}
h\Delta huv-h^{2}\dfrac{\partial w}{\partial t}uv-h(X_{1}+X_{2})(h)uv=-(%
\overline{\lambda }+\overline{\rho }\overline{R})h^{2}uv,
\end{equation}%
which imply $(iii)$.

In the equation (\ref{eq:42}), by following the same pattern as in the
Theorem \ref{teo:3.1}, we arrive that $(M_{1},g_{1},X_{1},\lambda _{1},\rho
_{1})$ is a $RBS$, where ${\lambda _{1}+\rho _{1}R_{1}=\overline{\lambda }+%
\overline{\rho }\overline{R}+\frac{n_{2}}{f}\sigma +\frac{1}{h}\psi .}$

Moreover, in the equation (\ref{eq:43}), if $X_{2}$ is a Killing vector
field and $\overline{\mathrm{Hess}}h=\psi \overline{g}$, we obtain that $%
M_{2}$ is an Einstein manifold, which completes the proof.
\end{proof}

Now, as an application of Theorem \ref{teo:3.4}, Theorem \ref{teo:3.5} and
Theorem \ref{teo:3.6}, we can give the following results:

\begin{thm}
\label{teo:4.2} Let $\overline{M}=(M_{1}\times _{f}M_{2})\times _{h}I$ be a
sequential standard static space-time and $(\overline{M},\overline{g},%
\overline{X},\overline{\lambda },\overline{\rho })$ a $RBS$ with $\overline{X%
}=X_{1}+X_{2}+w\partial _{t}$ where $X_{i}\in \mathfrak{X}(M_{i})$ for $%
1\leq i\leq 2$ and $w\partial _{t}\in \chi (I)$. Assume that $\overline{X}$
is a conformal vector field on $\overline{M}$. If $\mathrm{Hess}f=\sigma
\overline{g}$ and $\overline{\mathrm{Hess}}h=\psi \overline{g},$ then $M_{1}$
and $M_{2}$ are Einstein manifolds with factors ${\mu _{1}=-\frac{\Delta h}{h%
}+\frac{n_{2}}{f}\sigma +\frac{1}{h}\psi }$ and ${\mu _{2}=-\frac{\Delta h}{h%
}f^{2}+f^{\sharp }+\frac{1}{h}\psi f^{2}}$, respectively.
\end{thm}

\begin{proof}
Assume that $(\overline{M},\overline{g},\overline{X},\overline{\lambda },%
\overline{\rho })$ is a $RBS$ and $\overline{X}$ is a conformal vector field
on $\overline{M}$ with factor $2\alpha $. Then $\alpha $ is a constant and
\begin{equation*}
\overline{\mathrm{Ric}}(\overline{Y},\overline{Z})=(\overline{\lambda }+%
\overline{\rho }\overline{R}-\alpha )\overline{g}(Y,Z).
\end{equation*}%
If $\mathrm{Hess}f=\sigma \overline{g}$ and $\overline{\mathrm{Hess}}h=\psi
\overline{g}$, the above equation turns into%
\begin{equation*}
\mathrm{Ric}^{1}(Y_{1},Z_{1})-\frac{n_{2}}{f}\sigma g_{1}(Y_{1},Z_{1})-\frac{%
1}{h}\psi g_{1}(Y_{1},Z_{1})+\mathrm{Ric}^{2}(Y_{2},Z_{2})-f^{\sharp
}g_{2}(Y_{2},Z_{2})
\end{equation*}%
\begin{equation*}
-\frac{1}{h}\psi f^{2}g_{2}(Y_{2},Z_{2})+h\Delta huv
\end{equation*}%
\begin{equation*}
=(\overline{\lambda }+\overline{\rho }\overline{R}-\alpha
)g_{1}(Y_{1},Z_{1})+(\overline{\lambda }+\overline{\rho }\overline{R}-\alpha
)f^{2}g_{2}(Y_{2},Z_{2})-(\overline{\lambda }+\overline{\rho }\overline{R}%
-\alpha )h^{2}uv.
\end{equation*}%
Hence we find

\begin{equation*}
\mathrm{Ric}^{1}(Y_{1},Z_{1})={(\overline{\lambda }+\overline{\rho }%
\overline{R}-\alpha +\frac{n_{2}}{f}\sigma +\frac{1}{h}\psi )}%
g_{1}(Y_{1},Z_{1}),
\end{equation*}
\begin{equation*}
\mathrm{Ric}^{2}(Y_{2},Z_{2})={(\overline{\lambda }f^{2}+\overline{\rho }%
\overline{R}f^{2}-\alpha f^{2}+\frac{1}{h}\psi f^{2}+f^{\sharp })}%
g_{2}(Y_{2},Z_{2})
\end{equation*}
and $h\Delta huv=-(\overline{\lambda }+\overline{\rho }\overline{R}-\alpha
)h^{2}uv.$ So $M_{1}$ and $M_{2}$ are Einstein manifolds with factors ${\mu
_{1}=-\frac{\Delta h}{h}+\frac{n_{2}}{f}\sigma +\frac{1}{h}\psi }$ and ${\mu
_{2}=-\frac{\Delta h}{h}f^{2}+f^{\sharp }+\frac{1}{h}\psi f^{2}}$,
respectively.
\end{proof}

\begin{thm}
Let $\overline{M}=(M_{1}\times _{f}M_{2})\times _{h}I$ be a sequential
standard static space-time. Assume that $(\overline{M},\overline{g},%
\overline{X},\overline{\lambda },\overline{\rho })$ is a $RBS$ with $%
\overline{X}=X_{1}+X_{2}+w\partial _{t},$ where $X_{i}\in \mathfrak{X}%
(M_{i}) $ for $1\leq i\leq 2$ and $w\partial _{t}\in \chi (I).$ Then $(%
\overline{M},\overline{g})$ is Einstein if one of the following conditions
hold:

\begin{enumerate}
\item[(i)] $\overline{X}=w\partial _{t}$ and it is a Killing vector field on
$I$.

\item[(ii)] $X_1$ is a Killing vector field on $M_1$, $X_2$ and $%
w\partial_{t}$ are conformal vector fields on $M_2$ and $I$ with factors $%
-2X_1(\ln f)$ and $-2(X_1+X_2)(\ln h)$, respectively.

\item[(iii)] $X=X_{2}+w\partial _{t}$ and $X_{2},w\partial _{t}$ are Killing
vector fields on $M_{2}$ and $I$, respectively and $X_{2}(h)=0$.
\end{enumerate}
\end{thm}

\begin{thm}
Let $\overline{M}=(M_{1}\times _{f}M_{2})\times _{h}I$ be a sequential
standard static space-time and $(\overline{M},\overline{g},\overline{X},%
\overline{\lambda },\overline{\rho })$ a $RBS$ with $\overline{X}%
=X_{1}+X_{2}+w\partial _{t},$ where $X_{i}\in \mathfrak{X}(M_{i})$ for $%
1\leq i\leq 2$ and $w\partial _{t}\in \chi (I)$. Assume that $\mathrm{Hess}%
f=\sigma \overline{g}$ and $\overline{\mathrm{Hess}}h=\psi \overline{g}.$ If
$M_{1}$ and $M_{2}$ are Einstein manifolds, then $X_{1}$ and $X_{2}$ are
conformal vector fields on $M_{1}$ and $M_{2}$, respectively.
\end{thm}

\begin{proof}
Let $(\overline{M},\overline{g},\overline{X},\overline{\lambda },\overline{%
\rho })$ be a $RBS$ and $M_{1}$, $M_{2}$ Einstein manifolds with factors $%
\mu _{1}$ and $\mu _{2}$, respectively. If $\mathrm{Hess}f=\sigma \overline{g%
}$ and $\overline{\mathrm{Hess}}h=\psi \overline{g}$, then from the equation
(\ref{eq:41}), we can write
\begin{eqnarray*}
&&\mu _{1}g_{1}(Y_{1},Z_{1})-\frac{n_{2}}{f}\sigma g_{1}(Y_{1},Z_{1})-\frac{1%
}{h}\psi g_{1}(Y_{1},Z_{1})+\mu _{2}g_{2}(Y_{2},Z_{2})-f^{\sharp
}g_{2}(Y_{2},Z_{2})\hspace{2cm} \\
&&-\frac{1}{h}\psi f^{2}g_{2}(Y_{2},Z_{2})+h\Delta huv+\dfrac{1}{2}\mathcal{L%
}_{X_{1}}^{1}g_{1}(Y_{1},Z_{1})\hspace{2cm} \\
&&+\dfrac{1}{2}f^{2}\mathcal{L}_{X_{2}}^{2}g_{2}(Y_{2},Z_{2})-h^{2}\frac{%
\partial w}{\partial t}uv+fX_{1}(f)g_{2}(Y_{2},Z_{2})-uvh(X_{1}+X_{2})(h)%
\hspace{2cm} \\
&=&(\overline{\lambda }+\overline{\rho }\overline{R})g_{1}(Y_{1},Z_{1})+(%
\overline{\lambda }+\overline{\rho }\overline{R})f^{2}g_{2}(Y_{2},Z_{2})-(%
\overline{\lambda }+\overline{\rho }\overline{R})h^{2}uv.\hspace{2cm}.
\end{eqnarray*}%
Hence we have,
\begin{equation*}
\mathcal{L}_{X_{1}}^{1}g_{1}(Y_{1},Z_{1})=2(\overline{\lambda }+\overline{%
\rho }\overline{R}-\mu _{1}+\frac{n_{2}}{f}\sigma +\frac{1}{h}\psi
)g_{1}(Y_{1},Z_{1}),
\end{equation*}%
\begin{equation*}
\mathcal{L}_{X_{2}}^{2}g_{2}(Y_{2},Z_{2})=\frac{2}{f^{2}}((\overline{\lambda
}+\overline{\rho }\overline{R})f^{2}-\mu _{2}+f^{\sharp }+\frac{1}{h}\psi
f^{2}-fX_{1}(f))g_{2}(Y_{2},Z_{2})
\end{equation*}%
and
\begin{equation*}
h\Delta h-h^{2}\frac{\partial w}{\partial t}-uvh(X_{1}+X_{2})(h)=-(\overline{%
\lambda }+\overline{\rho }\overline{R})h^{2},
\end{equation*}%
which imply that $X_{1}$ and $X_{2}$ are conformal vector fields on $M_{1}$
and $M_{2}$, respectively.
\end{proof}

Now we consider a $RBS$\ with the structure of the sequential generalized
Robertson- Walker space-times. Firstly we define the notion of the
sequential generalized Robertson-Walker space-time.

Let $(M_{i},g_{i})$ be semi-Riemannian manifolds, $2\leq i\leq 3,$ and $%
f:I\longrightarrow \mathbb{R}^{+},$ $h:I\times M_{2}\longrightarrow \mathbb{R%
}^{+}$ two smooth functions. The $(n_{2}+n_{3}+1)$- dimensional \textit{%
sequential generalized Robertson-Walker space-time} $\overline{M}$ is the
triple product manifold $\overline{M}=I\times _{f}M_{2}\times _{h}M_{3}$
endowed with the metric tensor $\overline{g}=(-dt^{2}\oplus
f^{2}g_{2})\oplus h^{2}g_{3}$. Here $I$ is an open, connected subinterval of
$\mathbb{R}$ and $dt^{2}$ is the usual Euclidean metric tensor on $I$ \cite%
{dsu}.

\begin{prop}
\cite{dsu} \label{dsu:p41} Let $(\overline{M}=(I \times_{f}M_{2})\times
_{h}M_{3},\overline{g})$ be a sequential generalized Robertson-Walker
space-time and $X_{i},Y_{i}\in \mathfrak{X}(M_{i})$ for $2\leq i\leq 3$. Then

\begin{enumerate}
\item $\overline{\nabla} _{\partial_{t}}\partial_{t}= 0$

\item $\overline{\nabla}_{\partial_{t}}X_i=\nabla _{X_i}\partial_{t}=\frac{%
\dot{f}}{f}X_{i}$, $i=2,3$

\item $\overline{\nabla}_{X_2}Y_2=\nabla_{X_2}^2Y_2-f\dot{f}%
g_2(X_2,Y_2)\partial_{t}$,

\item $\overline{\nabla} _{X_{2}}X_{3}=\overline{\nabla}
_{X_{3}}X_{2}=X_{2}(\ln h)X_{3}$,

\item $\overline{\nabla} _{X_{3}}Y_{3}=\overline{\nabla}
_{X_{3}}^{3}Y_{3}-hg_{3}(X_{3},Y_{3})grad h$,
\end{enumerate}
\end{prop}

\begin{prop}
\cite{dsu}\label{dsu:p43} Let $(\overline{M}=(I\times_{f}M_{2})\times
_{h}M_{3},\overline{g})$ be a sequential generalized Robertson-Walker
space-time and $X_{i},Y_{i}\in \mathfrak{X}(M_{i})$ for $2\leq i\leq 3$. Then

\begin{enumerate}
\item $\overline{\mathrm{Ric}}(\partial _{t},\partial _{t})={\frac{n_{2}}{f}%
\ddot{f}+\frac{n_{3}}{h}\frac{\partial ^{2}h}{\partial t^{2}}}$

\item $\overline{\mathrm{Ric}}(X_{2},Y_{2})={\mathrm{Ric}%
^{2}(X_{2},Y_{2})-f^{\diamond }g_{2}\left( X_{2},Y_{2}\right) -\frac{n_{3}}{h%
}\overline{\mathrm{Hess}}h(X_{2},Y_{2})}$

\item $\overline{\mathrm{Ric}}(X_3,Y_3)=$ $\mathrm{Ric}^3(X_3,Y_3)-h^{%
\sharp}g_{3}(X_3,Y_3)$,

\item $\overline{\mathrm{Ric}}(X_{i},Y_{j})=0$ when $i\neq j,$ where $%
f^{\diamond }=-f\ddot{f}+(n_{2}-1)\dot{f}^{2}$ and $h^{\sharp }=h\Delta
h+(n_{3}-1)\left\Vert gradh\right\Vert ^{2}$.
\end{enumerate}
\end{prop}

By using of Lemma \ref{dsu:p32}, it is easy to state the following Corollary:

\begin{cor}
\label{sgrw} Let $(\overline{M}=(I\times _{f}M_{2})\times _{h}M_{3},%
\overline{g})$ be a sequential generalized generalized Robertson-Walker
space-time. Then
\begin{eqnarray*}
\mathcal{L}_{\overline{X}}\overline{g}(\overline{Y},\overline{Z}) &=&-2\frac{%
\partial w}{\partial t}uv+f^{2}\left( \mathcal{L}_{X_{2}}^{2}g_{2}\right)
(Y_{2},Z_{2})+h^{2}\left( \mathcal{L}_{X_{3}}^{3}g_{3}\right) (Y_{3},Z_{3})+
\\
&&+2wf\frac{\partial f}{\partial t}g_{2}(Y_{2},Z_{2})+2wh(\frac{\partial h}{%
\partial t}+X_{2}(h))g_{3}(Y_{3},Z_{3}),
\end{eqnarray*}%
where $\overline{X}=w\partial _{t}+X_{2}+X_{3}$, $\overline{Y}=u\partial
_{t}+Y_{2}+Y_{3}$ and $\overline{Z}=v\partial _{t}+Z_{2}+Z_{3}\in \chi (%
\overline{M})$.
\end{cor}

First, we give the following theorem as an application of Theorem \ref%
{teo:3.1}

\begin{thm}
Let $\overline{M}=(I\times _{f}M_{2})\times _{h}M_{3}$ be a sequential
generalized Robertson-Walker space-time. Assume that $(\overline{M},%
\overline{g},\overline{X},\overline{\lambda },\overline{\rho })$ is a $RBS$
with $\overline{X}=w\partial _{t}+X_{2}+X_{3}$ on $\overline{M},$ where $%
X_{i}\in \mathfrak{X}(M_{i})$ for $2\leq i\leq 3$ and $w\partial _{t}\in
\chi (I)$. Then

\begin{enumerate}
\item[(i)] ${-\frac{n_{2}}{f}\ddot{f}-\frac{n_{3}}{h}\frac{\partial ^{2}h}{%
\partial t^{2}}+\frac{\partial w}{\partial t}=\overline{\lambda }+\overline{%
\rho }\overline{R}}$,

\item[(ii)] When $\overline{\mathrm{Hess}}h=\psi \overline{g}$, $%
(M_{2},g_{2},f^{2}X_{2},\lambda _{2},\rho _{2})$ is a $RBS$, \newline
where ${\lambda _{2}+\rho _{2}R_{2}=\overline{\lambda }f^{2}+\overline{\rho }%
\overline{R}f^{2}+f^{\diamond }-wf\dot{f}+\frac{n_{3}}{h}\psi }$.

\item[(iii)] $(M_{3},g_{3},h^{2}X_{3},\lambda _{3},\rho _{3})$ is a $RBS$,
\newline
where ${\lambda _{3}+\rho _{3}R_{3}=\overline{\lambda }h^{2}+\overline{\rho }%
\overline{R}h^{2}+h^{\sharp }-wh\frac{\partial h}{\partial t}-whX_{2}(h)}.$
\end{enumerate}
\end{thm}

\begin{proof}
Assume that $(\overline{M},\overline{g},\overline{X},\overline{\lambda },%
\overline{\rho })$ is a $RBS$ soliton with the structure of the generalized
Robertson-Walker space-time $\overline{M}=(I\times _{f}M_{2})\times
_{h}M_{3} $. By Proposition \ref{dsu:p43} and Corollary \ref{sgrw}, the
proof is clear.
\end{proof}

The next result can be considered as a consequence of Theorem \ref{teo:3.4}.

\begin{thm}
Let $\overline{M}=(I\times _{f}M_{2})\times _{h}M_{3}$ be a sequential
generalized Robertson-Walker space-time and $(\overline{M},\overline{g},%
\overline{X},\overline{\lambda },\overline{\rho })$ a $RBS$ soliton with $%
\overline{X}=w\partial _{t}+X_{2}+X_{3}$. Assume that $\overline{X}$ is a
conformal vector field on $\overline{M}$. If $\overline{\mathrm{Hess}}h=\psi
\overline{g},$ then $M_{2}$ and $M_{3}$ are Einstein manifolds with factors $%
{\mu _{1}=(-\frac{n_{2}}{f}\ddot{f}-\frac{n_{3}}{h}\frac{\partial ^{2}h}{%
\partial t^{2}})f^{2}+f^{\diamond }+\frac{n_{3}}{h}\psi }$ and ${\mu _{2}=(-%
\frac{n_{2}}{f}\ddot{f}-\frac{n_{3}}{h}\frac{\partial ^{2}h}{\partial t^{2}}%
)h^{2}+h^{\sharp }}$, respectively.
\end{thm}

\begin{proof}
The proof is similar to the proof of Theorem \ref{teo:3.4} and Theorem \ref%
{teo:4.2}.
\end{proof}

Now, we give the following result for gradient $RBS$ with the structure of
the generalized Robertson-Walker space-time.

\begin{thm}
Let $(\overline{M}=(I\times _{f}M_{2})\times _{h}M_{3},\overline{g},\nabla u,%
\overline{\lambda },\overline{\rho })$ be a sequential generalized
Robertson-Walker space-time and $(\overline{M},\overline{g},\nabla u,%
\overline{\lambda },\overline{\rho })$ a $RBS$, where
\begin{equation*}
u=\int_{a}^{t}f(r)dr,\hspace{1cm}\text{for some constant }\hspace{0cm}a\in I
\end{equation*}%
then $\overline{M}$ is an Einstein manifold with factor $(\overline{\lambda }%
+\overline{\rho }\overline{R}-\dot{f}).$
\end{thm}

\begin{proof}
Suppose that $X=\nabla u$. Then $X=f(t)\partial _{t}$. \newline
Let $\{\partial _{t},\partial _{1},\partial _{2},\ldots ,\partial
_{n_{2}},\partial _{n_{2}+1},\ldots ,\partial _{n_{2}+n_{3}}\}$ be an
orthonormal basis for $\chi (\overline{M})$. The Hessian of $u$ is given by $%
Hessu(Y,Z)=\overline{g}(\overline{\nabla }_{Y}\nabla u,Z).$ Here, we have
the following six cases:

\noindent i) When $X=Y=\partial _{t}$, we get
\begin{eqnarray*}
Hessu(\partial _{t},\partial _{t}) &=&\overline{g}(\overline{\nabla }%
_{\partial _{t}}\nabla u,\partial _{t}) \\
&=&\dot{f}\overline{g}(\partial _{t},\partial _{t})
\end{eqnarray*}

\noindent ii) When $Y=\partial _{t}$ and $Z=\partial _{i}$, $1\leq i\leq
n_{2}$, we have
\begin{eqnarray*}
Hessu(\partial _{t},\partial _{i}) &=&\overline{g}(\overline{\nabla }%
_{\partial _{t}}\nabla u,\partial _{i}) \\
&=&\dot{f}\overline{g}(\partial _{t},\partial _{i}).
\end{eqnarray*}

\noindent iii) When $Y=\partial _{t}$ and $Z=\partial _{k}$, $n_{2}+1\leq
k\leq n_{2}+n_{3}$, $Hessu=\dot{f}\overline{g}$.

\noindent iv) When $Y=\partial _{i}$ and $Z=\partial _{j}$, $1\leq i,j\leq
n_{2}$, we have%
\begin{eqnarray*}
Hessu(\partial _{i},\partial _{j}) &=&\overline{g}(\overline{\nabla }%
_{\partial _{i}}\nabla u,\partial _{j}) \\
&=&f\overline{g}(\frac{\dot{f}}{f}\partial _{i},\partial _{j}) \\
&=&\dot{f}\overline{g}(\partial _{i},\partial _{j}).
\end{eqnarray*}

\noindent v) When $Y=\partial _{i}$, $1\leq i\leq n_{2}$ and \noindent $%
Z=\partial _{k}$, $n_{2}+1\leq k\leq n_{2}+n_{3}$, $Hessu=\dot{f}\overline{g}%
.$

\noindent vi) Finally, when $Y=\partial _{k}$ and $Z=\partial _{l}$, $%
n_{2}+1\leq k,l\leq n_{2}+n_{3}$,
\begin{eqnarray*}
Hessu(\partial _{k},\partial _{l}) &=&\overline{g}(\overline{\nabla }%
_{\partial _{k}}\nabla u,\partial _{l}) \\
&=&f\overline{g}(\frac{\dot{f}}{f}\partial _{k},\partial _{l}) \\
&=&\dot{f}\overline{g}(\partial _{k},\partial _{l}).
\end{eqnarray*}%
Hence, $Hessu(Y,Z)=\dot{f}\overline{g}(Y,Z)$ and $\mathcal{L}_{X}\overline{g}%
(Y,Z)=2Hessu(Y,Z)=2\dot{f}\overline{g}(Y,Z)$. Therefore, $\overline{\mathrm{%
Ric}}=(\overline{\lambda }+\overline{\rho }\overline{R}-\dot{f})\overline{g}$
is satisfied, which completes the proof.
\end{proof}

\medskip

\noindent Dilek A\c{C}IKG\"{O}Z KAYA\newline
Ayd\i n Adnan Menderes University, \newline
Department of Mathematics, \newline
09010, Ayd\i n, T\"{U}RK\.{I}YE \newline
Email:dilek.acikgoz@adu.edu.tr

\medskip

\noindent Cihan \"{O}ZG\"{U}R \newline
Department of Mathematics, \newline
Izmir Democracy University, \newline
35140, Karaba\u{g}lar, \.{I}zmir, T\"{U}RK\.{I}YE \newline
Email: cihan.ozgur@idu.edu.tr

\end{document}